\documentclass[options]{amsart}
 \pdfoutput=1
\usepackage{color}
\usepackage{latexsym,ifthen,amssymb,amsmath}
\usepackage{hyperref}
\usepackage{cleveref}

\newtheorem{theorem}{Theorem}[section]
\newtheorem{lemma}[theorem]{Lemma}

\newtheorem{claim}[theorem]{Claim}
\newtheorem{corollary}[theorem]{Corollary}
\newtheorem*{Ramsey's theorem}{Ramsey's theorem}

\theoremstyle{definition}
\newtheorem{definition}[theorem]{Definition}

\newcommand{\Cross}{\mathrm{Cross}}

\DeclareMathOperator{\Fin}{\mathsf{Fin}}
\DeclareMathOperator{\R}{\mathcal{R}} 

\theoremstyle{remark}

\newtheorem {question}[theorem]{Question}
\numberwithin{equation}{section}

\newcommand{\uhr}{\upharpoonright}
\newcommand{\restrict}{\upharpoonright}
\def\mcal{\mathcal}
\def\t{\tilde}
\def\h{\hat}

\title{Extracting randomness within a subset is hard}


\author{Bj\o rn Kjos-Hanssen}
\email{bjoern.kjos-hanssen@hawaii.edu}
\address{Department of Mathematics, University of Hawai\textquoteleft i at M\=anoa, Honolulu, HI 96822, USA}

\author{Lu Liu }
\email{g.jiayi.liu@gmail.com}
\address{Department of Mathematics and Statistics, Central South University,
Changsha, 410083, China}

\thanks{This work was partially supported by a grant from the Simons Foundation (\#315188 to Bj\o rn Kjos-Hanssen).
Lu Liu is partially supported by Natural Science Foundation of Hunan Province of China
2018JJ3623.}

\begin{document}
	\begin{abstract}
	The tree forcing method of Liu enables the cone avoiding
	of bounded enumeration of a given tree,
	within subsets or co-subsets
	of an \emph{arbitrary} given set,
	provided the given tree
	does not admit computable bounded enumeration.
	Using this result, he settled and reproduced
	a series of problems and results in reverse mathematics
	and the theory of algorithmic randomness, including showing that every 1-random set has an infinite subset or co-subset which computes no 1-random set.

	In this paper,
	we show that for any given 1-random set $A$, there
	exists an infinite subset $G$ of $A$ such that $G$ does
	not compute any set with positive effective Hausdorff
	dimension.

	In particular we answer in the affirmative Kjos-Hanssen's 2006 question whether each 1-random set has an infinite subset which computes no 1-random
	set.

	The result is surprising in that
	the tree forcing technique
	seems to heavily
	rely on subset co-subset combinatorics,
	whereas this result does not.
	\end{abstract}

\maketitle

\section{Introduction}
	Computability theory aims to classify real numbers, or equivalently infinite binary sequences, by their relative computational power.
	This is done by means of several orderings, the most fundamental of which may be that of the Turing degrees.
	For instance, $\mathbf 0$ is the Turing degree of computable sequences, which are all regarded as trivial. Next, $\mathbf 0'$ is the Turing degree of the sequence of answers for the halting problem for Turing machines. It is also the Turing degree of many natural problems such as solvability of diophantine equations.

	On the other hand, if we choose the bits of our sequence randomly enough to have no computable pattern, roughly speaking, we get a collection of Turing degrees $\mathbf r$ that are called \emph{Martin-L\"of random}. Some of these are comparable with $\mathbf 0'$, but most are not. Such a degree $\mathbf r$ consists of a random sequence $R$, which can also be viewed as a set $R\subseteq\omega=\mathbb N$, together with all sequences that are computationally equivalent to $R$.
	The problem of computing a Martin-L\"of random set also arises in Reverse Mathematics in the guise of the formal system $\mathsf{WWKL}_0$ (Weak Weak K\"onig's Lemma), discussed below.

	Martin-L\"of random sets are also known as 1-random, and can be contrasted with ``more random'' sets (2-random and above) and ``less random'' sets.
	When can we get more randomness from less? This idea of \emph{extracting randomness} under various
	conditions has drawn attention from computability
	theorists. Existing results indicate that
	from a set with a low degree of randomness we cannot extract a set with a high degree
	of randomness. For example, Miller
	\cite{miller2011extracting}
	showed that there exists some set of effective Hausdorff
	dimension 1/2
	from which no 1-random set can be computed, thus
	separating the computability theoretic strengths
	(Muchnik degrees) of the two classes.
	Therefore, it is reasonable to believe that
	extracting randomness under various combinatorial
	conditions will also fail.
	In \cite{Kjos-Hanssen2009Infinite,kjos2011strong},
	Kjos-Hanssen studied the problem of extracting
	randomness within an infinite subset of a given
	1-random set. There he showed that
	every 2-random set admits an infinite subset that does not compute any
	1-random set. Later Liu \cite{liu2015cone}
	showed that every 1-random set admits an infinite subset or co-subset that does not compute any
	1-random set. Here we obtain the natural common strengthening of these results of Miller, Kjos-Hanssen, and Liu.

	Extracting randomness under various
	combinatorial conditions is also an interesting issue
	in reverse mathematics. In reverse mathematics,
	there are questions
	concerning whether an arithmetical statement
	implies $\mathsf{WWKL}_0$.
	$\mathsf{WWKL}_0$ is the statement that
	any positive measure binary tree admits a path.
	Or, roughly speaking, there exist 1-random
	sets.
	Proving that some
	arithmetical
	statement does not imply $\mathsf{WWKL}_0$
	involves
	constructing a set satisfying certain conditions
	while
	avoiding computing any 1-random set. \cite{Ambos-Spies2004Comparing}
	separates $\mathsf{DNR}$ from $\mathsf{WWKL}_0$.
	They construct a diagonal non recursive function
	that does not compute any 1-random set.
	This result was reproduced in \cite{liu2015cone}
	using another proof.
	\cite{liu2015cone} proved that
	if a tree does not admit bounded enumeration
	(see Definition \ref{boundenumeration}),
	then any given set $A$ admits an
	infinite
	subset or co-subset that also does not compute
	a bounded enumeration of that tree (thus does not compute
	a path of that tree). Because the tree
	defining $\mathsf{WWKL}_0$ does not admit computable
	bounded enumeration, the result
	therefore implies that
	$\mathsf{RT_2^2}$
	does not imply $\mathsf{WWKL}_0$.
	In addition to 1-randomness, many other randomness
	notions are also defined by trees,
	such as effective Hausdorff dimension. 
	As far as we know, there is
	no natural example of a tree that does not admit a computable
	path but admits a computable bounded enumeration.
	Therefore, generally speaking,
	the result of
	\cite{liu2015cone} means that extracting
	randomness under a subset co-subset condition is
	almost impossible.
	There is also ongoing research in reverse mathematics
	that seeks to construct computationally weak
	(in senses other than that of the inability to extract randomness)
	solutions
	under various conditions
	(other than subset co-subset)
	\cite{chong2014metamathematics,lerman2013separating,patey2015iterative,wang2014definability}.
	Such constructions generally yield
	conclusions of the form ``$\Gamma$ does not imply $\Psi$''.
	We are hopeful that our
	method can be adapted to construct computationally weak
	solutions
	of a random instance.

	In this paper, we adapt the proof
	in \cite{liu2015cone} to show that every 1-random
	set admits an infinite subset
	(instead of ``subset or co-subset'') that does not compute
	any 1-random set. Thus we answer a question of
	Kjos-Hanssen from the American Institute of Mathematics workshop ``Effective Randomness'' held in 2006. 

	The result is interesting
	because it seems that the combinatorial argument in
	\cite{liu2015cone} relies heavily on the fact that
	$A$ and $\overline{A}$ form a partition of $\omega$.

	We end this section by
	giving some definitions and the
	main result.
	In the following, we fix
	a universal prefix free machine
	$U$ and let $K_U(\rho)$ denote
	the corresponding Kolmogorov complexity
	of $\rho\in 2^{<\omega}$. For $X\in 2^\omega$
	we write
	$X\uhr M$ to denote the initial segment of $X$
	of length $M$.
	\begin{definition}
		For a set $A$, we say
		$A$ is \emph{effectively compressible} iff there exists a
		computable function $f:\omega\rightarrow\omega$ such that
		$K_U(A\upharpoonright f(n))\leq f(n) - n$.
	\end{definition}

	For any set $Q$, let $\Fin(Q)$ be the set of finite subsets of $Q$.
	The set $j^n$ consists of all functions $\sigma:n\to j$, where $j=\{0,1,\dots,j-1\}$.
	A set $A$ is \emph{c.e.}~if it is computably enumerable, and \emph{co-c.e.}~if its complement is c.e.
	\begin{definition}[ Beigel et al.~\cite{beigel2006enumerations}]
	\label{boundenumeration}
		Given a set $S\subseteq
		j^{<\omega}$,
		an $l$-\emph{enumeration} of
		$S$ is a function $g:\omega\rightarrow
		\Fin(j^{<\omega})$ such that
		$|g(n)|\leq l$ and $g(n)\cap S\cap j^n\ne\emptyset$
		for all $n$.
		A \emph{bounded enumeration} of
		$S$ is an $l$-enumeration for some $l\in\omega$.
		If $D\subseteq\omega$, we say that $S$ admits a $D$-computable $l$-enumeration (resp.~bounded enumeration) if
		there is a $D$-computable function that is an $l$-enumeration (resp.~bounded enumeration) of $S$.
	\end{definition}

	\begin{theorem}\label{th1}
	Given a set $A$ that is not effectively compressible,
	let $S^u\subseteq j^{<\omega}$,
	$u\in\omega$, be
	a family of co-c.e. sets such that
	none of the $S^u,u\in\omega$ admits
	computable bounded enumeration.
	Then there exists an infinite subset
	of $A$, namely $G$, such that none of the
	$S^u$
	admits bounded enumeration computable in $G$.
	\end{theorem}

	\begin{corollary}\label{coro}
		For any 1-random set $A$, there exists an infinite subset
		of $A$, namely $G$, such that $G$ does not compute any
		set with positive effective Hausdorff dimension.
	\end{corollary}
	\begin{proof}
		The corollary follows by noting that the sequence
		of trees defining ``positive effective Hausdorff dimension''
		does not admit a computable
		bounded enumeration,
		that 1-randomness implies not being
		effectively compressible, and that the infinite subset $G$ of $A$
		gives the bounded enumeration (in fact, 1-enumeration) given by $g(n)=\{G\uhr n\}$.

		To see the former, note that the trees can be taken to be
		\[
			T_{m, c}=\{\sigma\in 2^{<\omega}: (\forall k<|\sigma|)\, K(\sigma\restrict k)\ge k/m-c\},\quad m>0, c,m\in\omega.
		\]
		as $X$ has positive effective Hausdorff dimension if and only if
		\[
			\exists m,c\quad\forall n\quad X\restrict n\in T_{m,c}.
		\]
		If $T_{m,c}$ has a computable $l$-enumeration then we can describe $\sigma\in T_{m,c}$
		by giving $n=|\sigma|$ and the index of $\sigma$ in the list of up to $l$ elements of $S\cap 2^n$.
		This would show $K(\sigma)\le^+ 2(\log n + \log l)$ which for large $n$ contradicts $K(\sigma)\ge |\sigma|/m-c$.
	\end{proof}

	The \emph{complex packing dimension} of a set $A$ was defined in \cite{kjos2013randomness} to be the supremum of $\inf_{n\in M}K(A\restrict n)/n$ over all infinite computable sets $M$. In \cite{kjos2013randomness} sets of positive complex packing dimension were shown to be computationally weak, in that we cannot uniformly extract a stochastically bi-immune set from them. Here we obtain a result showing sets of positive complex packing dimension are computationally strong, or at least hard to compute.

	\begin{corollary}\label{bjoernNew}
		For any 1-random set $A$, there exists an infinite subset
		of $A$, namely $G$, such that $G$ does not compute any
		set with positive complex packing dimension.
	\end{corollary}
	\begin{proof}
		Similarly to the proof of Corollary \ref{coro}: take the trees to be
		\[
			T_{m,M, c}=\{\sigma\in 2^{<\omega}: (\forall k\in M)\,K(\sigma\restrict k)\ge k/m-c\}, \quad m>0, c,m\in\omega, M\in\mathfrak C
		\]
		where $\mathfrak C$ is the collection of all infinite computable sets.
		These trees are all co-c.e. The fact that they are not \emph{uniformly} co-.c.e.~is not a problem since Theorem \ref{th1} uses a construction that deals with each tree in forcing requirements.
	\end{proof}

	The remainder of the paper
	is dedicated to proving Theorem
	\ref{th1}.
	In Section \ref{sec2}, we introduce the forcing conditions.
	We introduce the
	requirements in Section \ref{secframe}, and
	a general scheme of the proof
	is also described there. Section \ref{sec3} is devoted
	to the proof of the main lemma, Lemma \ref{lem3}.

	\subsection{Notation}
		We use $\Psi$, $\Psi_e$ to denote
		a Turing functional, the Turing functional
		with index $e$ respectively.
		As there is an effective bijection between $\omega$ and $\Fin[j^{<\omega}]$,
		we shall assume that for every Turing functional
		$\Psi$, every oracle $X$,
		there exists $l_\Psi\in\omega$
		such that
		$\Psi^X$ is computing an $l_\Psi$-enumeration
		of $j^{<\omega}$ and $\Psi^X(n)\downarrow\rightarrow
		\Psi^X(n)\subseteq j^n$ for all $n\in\omega$.
		So
		$\rho\in\Psi(n)$ means that $\rho$
		is an element of $\Psi(n)$;
		$\Psi(n)\cap T$ refers to the intersection
		of the set $\Psi(n)$ and $T$.
		We sometimes regard a 0-1 sequence
		$\rho\in 2^{<\omega}$ or
		$X\in 2^\omega$
		as a set of integers and use
		$\rho\subseteq \tau$ to denote
		a set $\rho$ being a subset of $\tau$;
		$\rho\cap \tau$ to denote the string
		with $(\rho\cap \tau)(i) = \rho(i)\cdot \tau(i)$.
		For $\rho\in 2^{<\omega}$, $|\rho|$
		refers to the length of $\rho$.
		For $\rho\in 2^{<\omega}$,
		$\tau\in 2^{<\omega}$ or $\tau\in 2^\omega$,
		we use $\rho\prec\tau$ to denote
		$\tau$ being an extension of $\rho$;
		we write $\tau/\rho$ for the string obtained
		by replacing the first $|\rho|$ bits of
		$\tau$ by $\rho$;
		we use $\bar{\rho}$ to denote the string
		$(1-\rho(0))(1-\rho(1))\cdots$;
		we use $\rho\uhr_{a}^b$ to denote the binary string
		with $\rho\uhr_a^b(i) = \rho(i+a)\wedge
		\ |\rho\uhr_a^b\ |=b-a+1$.
		For a co-c.e. or c.e. set $W$, $W[t]$
		denote the set computed by time $t$.

\section{Forcing conditions }\label{sec2}
	We begin by reviewing Mathias forcing and the forcing conditions
	introduced in \cite{Liu2012RT22,liu2015cone}.
	We say $X\in 2^\omega$ is a \emph{$k$-partition} iff
	\begin{itemize}
		\item $X=X_0\oplus X_1\oplus\cdots\oplus X_{k-1}$;
		\item $\bigcup_{i=0}^{k-1}X_i=\omega$.
	\end{itemize}

	A class
	$Q\subseteq 2^\omega$
	is a \emph{$k$-partition class} iff
	for every $ X\in Q$, $X$ is a
	$k$-partition.

	\begin{definition}[Mathias condition]
		\label{def2}
		A \emph{Mathias condition} is a pair $(\sigma,X)$
		with $\sigma \in 2^{<\omega}$ and $X \in 2^\omega$.

		We say that $(\tau,Y)$ \emph{extends} the Mathias condition $(\sigma,X)$
		iff $\sigma \preceq \tau$ and $Y/\tau \subseteq X/\sigma$.
		Write $(\tau,Y)\leq (\sigma,X) $ to denote
		the extension relation.

		We say that a set $G$ \emph{satisfies} the Mathias condition
		$(\sigma,X)$ if $\sigma \prec G$
		and $G \subseteq X/\sigma$.
	\end{definition}

	\begin{definition}[Tree forcing conditions]
		The forcing conditions we use
		to construct $G$ are tuples
		$(k,\sigma_0,\ldots,\sigma_{k-1},Q)$,
		where $k>0$, $\sigma_i \in
		2^{<\omega}$, and $\sigma_i\subseteq A$ for
		all $i\leq k-1$,
		and $Q$ is a nonempty $\Pi^{0}_1$ $k$-partition class.
		Moreover, for every $X_0 \oplus \cdots \oplus X_{k-1} \in Q$ and
		every $i\leq k-1$, $\sigma_i\subseteq X_i\cap A$.
		We regard each $X_0 \oplus \cdots \oplus X_{k-1} \in Q$
		s representing $k$ many
		Mathias conditions $(\sigma_i,X_i)$, $i < k$.
	\end{definition}

	\begin{definition}
		\label{extend}
		We say that a condition
		$d'=(k',\sigma_0',\ldots,\sigma_{k'-1}',Q')$ \emph{extends}
		a condition
		$d=(k,\sigma_0,\ldots,\sigma_{k-1},Q)$,
		(henceforth $d'\leq d$), if there is a function $f : k'
		\rightarrow k$ such that
		\[
			\forall i<k'\ \forall Y_0 \oplus
			\cdots \oplus Y_{k'-1} \in Q'
			\ \exists X_0 \oplus \cdots \oplus
			X_{k-1} \in Q \ \big[
			(\sigma_i',Y_i)\leq (\sigma_{f(i)},X_{f(i)})\big].
		\]
		In this case, we say that
		\begin{itemize}
			\item$f$ witnesses the extension $d'\le d$;
			\item part $i$ of
			the condition $d'$ refines part $f(i)$ of
			the condition $d$.\footnote{Strictly speaking we have not defined ``part $i$''. We could also say: $d'$ $f$-refines $d$.}
		\end{itemize}
	\end{definition}

	\begin{definition}
		\label{def-sat}
		We say that a set
		$G$ \emph{satisfies}
		condition $(k,\sigma_0,\ldots,\sigma_{k-1},Q)$
		iff there is an $X_0 \oplus
		\cdots \oplus X_{k-1} \in Q$ such that
		$G$ satisfies some $(\sigma_i,X_i)$.
		In this case, we also say that $G$
		satisfies $(k,\sigma_0,\ldots,\sigma_{k-1},Q)$
		\emph{on part $i$}.
	\end{definition}
	We assume that for each Turing functional $\Psi$
	there exists $l_\Psi$ depending on $\Psi$
	such that for every $X$, $\Psi^X$ is
	an $l_\Psi$-enumeration with $\Psi^X(m)\downarrow\rightarrow
	\Psi^X(m)\subseteq j^m$.

	For each Turing functional $\Psi$ and $u\in\omega$, we need to satisfy
	the \emph{requirement} $\R_{\Psi}^u$:
	\[
		\tag{$\R_{\Psi}^u$}
		\text{$\Psi^G$ is not an $l_\Psi$-enumeration of $S^u$ if
		$\Psi^G$ is total.}
	\]
	\begin{definition}\label{def-force}
		We say
		condition $d$ \emph{forces} requirement
		$\mathcal{R}$ \emph{on part $i$} iff every $G$
		satisfying $d$ on part $i$
		also satisfies requirement
		$\mathcal{R}$.
		We say condition $d$ \emph{forces}
		requirement $\mathcal{R}$
		iff it forces $\mathcal{R}$ on all
		parts.
	\end{definition}

	\begin{definition}
		We say part $i$ of condition
		$c=(k,\sigma_0,\ldots,\sigma_{k-1},Q)$ is \emph{acceptable}
		iff there exists
		$X_0
		\oplus \cdots \oplus X_{k-1} \in Q$
		such that $X_i \cap A$
		is infinite, where $A$ is the set given in Theorem \ref{th1}.
	\end{definition}
\section{Frame of the proof}\label{secframe}

	We will construct an infinite subset
	$G$ of $A$ satisfying all requirements
	$\R_{\Psi}^u$, using the following lemma.
	\begin{lemma}\label{bjoernAdd}
		Suppose there exists a sequence of conditions
		\[
			d_0\geq d_1\geq\cdots \geq d_s\geq\cdots,
		\]
		$d_s = (k_s,\sigma^s_{0},
		\cdots,\sigma^s_{k_s-1},Q_s)$, with
		$Q_s\neq \emptyset$
		such that
		\begin{itemize}
			\item for every $\Psi$, $u\in\omega$,
		$d_{\langle\Psi,u\rangle}$ forces
		$\mathcal{R}_\Psi^u$, and
			\item for all $s$, $|\{n:\sigma_i^s(n)=1\}|\geq s$
			for all $i$ in the acceptable parts
			of $d_s$.
		\end{itemize}
		Then there exists an infinite subset
		$G$ of $A$ satisfying all requirements
		$\R_{\Psi}^u$.
	\end{lemma}
	\begin{proof}
		Note that if such a sequence of condition exists,
		then the initial segments of acceptable parts of each
		condition $d_s$ forms a tree $\mathcal{T}$:
		the nodes in the $s^{th}$ level
		are $\sigma_i^s$, with $i$ being an acceptable parts
		of $d_s$; the predecessor of $\sigma_i^s$ is
		$\sigma^{s-1}_{f_s(i)}$ where $f_s$ witnesses that
		$d_{s-1}\geq d_s$.
		Obviously, every condition $d_s$
		admits some acceptable part since $Q_s$
		is a partition class.
		Therefore $\mcal{T}$ is finitely branching
		and infinite. Thus there is an infinite
		path through $\mcal{T}$,
		namely $\sigma^s_{i_s},s\in\omega$.
		By the definition of extension,
		$\sigma^{s+1}_{i_{s+1}}\succeq \sigma^{s}_{i_{s}}$,
		so $G=\cup_s \sigma^s_{i_s}$ is well-defined.
		By the definition of condition,
		$G\subseteq A$. Since $i_s$ is an
		acceptable part of $d_s$, $|\{n:\sigma^s_{i_s}(n)\}|\ge s$.
		Thus $G$ is infinite.
		Moreover, for each $Q_s$, by compactness,
		there exists $X_0\oplus X_1\oplus\cdots\oplus X_{k_s-1}\in Q_s$
		such that $G\subseteq X_{i_s}/\sigma^s_{i_s}$.
		To see this, fix an arbitrary $Q_s$,
		note that by the definition of extension,
		for any $s'>s$, the set $Q_{s,s'} = \big\{X_0\oplus\cdots\oplus X_{k_s-1}\in Q_s:
		\sigma^{s'}_{i_{s'}}\subseteq X_{i_s}/\sigma^s_{i_s}\big\}
		\ne\emptyset$ is a closed set and
		$Q_{s,s'+1}\subseteq Q_{s,s'}$.
		Thus $\cap_{s'>s} Q_{s,s'}\ne\emptyset$.
		So there exists $X_0\oplus\cdots\oplus X_{k_s-1}\in\cap_{s'>s} Q_{s,s'}
		\subseteq Q_s$ such that $G\subseteq X_{i_s}/\sigma^s_{i_s}$,
		i.e., $G$ satisfy part $i_s$ of condition $d_s$.
		Thus $G$ satisfies all requirements. 
	\end{proof}
	Now it remains to show that a sequence of conditions as in Lemma \ref{bjoernAdd} exists.
	First, we note that it is trivial to ensure
	that whenever $i$ is an acceptable part of $d_s$,
	then the initial segment of part $i$, namely $\sigma_i$,
	contains more than $s$ many elements.
	\begin{lemma}\label{lem1}
		For every condition $d=(k,\sigma_0,\cdots,\sigma_{k-1},Q)$
		and every
		$s\in \omega$, if $Q\neq \emptyset$,
		then there is a condition $d'\leq d$
		such that for every acceptable part $i$
		of $d'$, the initial segment $\sigma_i'$ of $d'$
		contains at least $s$ many elements.
	\end{lemma}
	\begin{proof}
		If $Q\neq \emptyset$, then
		$Q$ admits some acceptable part
		since
		$Q$ is a partition class.
		We simply extend
		each initial segment of condition
		$d$'s acceptable
		parts to
		include at least $s$ many elements
		until every initial segment of any
		acceptable part of the current
		condition $d'$ contains more than $s$ elements in $A$.
		\footnote{Note that
		after some
		extension of the initial segments of the other parts,
		an originally acceptable part may become unacceptable.
		So it is not necessary that all acceptable parts
		of $c$ are extended.}.
	\end{proof}
	Now it remains to show that every
	requirement $\R_\Psi^u$ can be forced
	by extending the condition.
	\begin{lemma}\label{lem2}
		Given any requirement
		$\R_\Psi^u$ and
		any condition $d=(k,\sigma_0,\cdots,\sigma_{k-1},Q)$,
		there is a
		condition
		$d'\leq d$ that forces $\mathcal{R}_\Psi^u$.
	\end{lemma}

	Lemma \ref{lem2}
	clearly follows from the following Lemma
	\ref{lem3}. For any condition
	$d$ let $U(d)$ denote the set of parts of $d$\footnote{``The set of parts of $d$ that do not force'' is shorthand for ``the set of all $i$ such that part $i$ of $d$ does not force''.}
	that do not force $\mathcal{R}_\Psi^u$.
	\begin{lemma}\label{lem3}
		Given any requirement $\mathcal{R}_\Psi^u$
		and any condition
		$d=(k,\sigma_0,\ldots,\sigma_{k-1},Q)$
		with $U(d)\ne\emptyset$,
		there exists a condition
		$d'=(k',\sigma_0',\cdots,\sigma_{k'-1}',Q')\leq d$,
		such that $|U(d')|<|U(d)|$.
	\end{lemma}

	The next section is devoted to
	the proof of Lemma \ref{lem3}.

\section{Proof of Lemma \ref{lem3}}\label{sec3}
	Fix the condition
	$d=(k,\sigma_0,\ldots,\sigma_{k-1},Q)$
	with $ U(d)\ne\emptyset$
	and the requirement
	$\mathcal{R}_\Psi^u$ given in Lemma \ref{lem3}.
	For any $\t{n}$,
	let
	\begin{equation}\label{pad}
		\sigma_i^{\t{n}} = \sigma_i\, 0^{\t{n}-|\sigma_i|}.
	\end{equation}
	Thus, we pad $\sigma_i$ with zeros to achieve length $\t{n}$.

	For any $m\in\omega$, $V\subseteq j^m$, and $\t n$, we define a class $Q_V^{\t n}$ by the condition that
	$\t{X}_0\oplus\cdots\oplus \t{X}_{2k-1}\in Q_V^{\t n}$ iff the following two conditions hold:
	\begin{enumerate}
		\item There exists $X_0 \oplus \cdots \oplus X_{k-1} \in Q$ with
			$X_i=\t{X}_{2i} \cup \t{X}_{2i+1}$ for all $i\leq k-1$;
		\item For each $i\in U(d)$ and each $\sigma'_{2i}\succeq \sigma^{\t{n}}_i$, $\sigma'_{2i+1}\succeq \sigma^{\t{n}}_i$,
			with $\sigma'_{2i}-\sigma^{\t{n}}_i\subseteq\t{X}_{2i}$ and $\sigma'_{2i+1}-\sigma^{\t{n}}_i\subseteq \t{X}_{2i+1}$,
			we have:
			\begin{eqnarray*}
				\Psi^{\sigma'_{2i}}(m)\downarrow&\rightarrow& \Psi^{\sigma'_{2i}}(m)\cap V\ne\emptyset,\\
				\Psi^{\sigma'_{2i+1}}(m)\downarrow&\rightarrow&\Psi^{\sigma'_{2i+1}}(m)\cap V\ne\emptyset.
			\end{eqnarray*}
	\end{enumerate}
	Thus in forming the class $Q_V^{\t n}$ we pad the strings in condition $d$ to achieve length $\t n$, we split the parts of $Q$ to form a $2k$-partition class from a $k$-partition class, and we force meeting of the set $V$.

	Note that for every $\t{n}$, $V$,
	\begin{itemize}
	\item
	$Q_V^{\t{n}}$ is a $\Pi_1^0$ $2k$-partition
	class;
	\item The set $\big\{V'\subseteq 2^{<\omega}: \text{for some } m\in\omega,
	V'\subseteq j^m; \text{ and }
	Q^{\t{n}}_{V'}\ne\emptyset
	\ \big\}$ is
	co-c.e. (uniformly in $\t{n}$).
	\end{itemize}


	\begin{definition}[Dispersedness]\label{disperse}
		A collection of sets $\{V_n\}_{n\leq N-1}$
		is $k$-\emph{dispersed}
		iff for every $k$-partition of
		$\{0,\cdots,N-1\}$,
		namely $W_0,\cdots,W_{k-1}$,
		there exists a part $W_{k'}$ such
		that $W_{k'}\ne\emptyset$
		and $\bigcap\limits_{n \in W_{k'}}V_n = \emptyset$.
	\end{definition}
		\footnote{
			In Definition \ref{disperse}, a $k$-partition of $N=\{0,\dots,N-1\}$ is a partition of $N$ into $k$ equivalence classes or blocks.
			It may be easier to consider the negation: a collection of sets is \emph{not} $k$-dispersed iff it can be partitioned into $k$ subcollections,
			each having nonempty intersection.
		}
	For every
	$\t{n}\geq \max\limits_{i\leq k-1}|\sigma_i|$,
	$m\in\omega$, consider
	the collection of
	clopen sets
	\[
		\text{Meetable}_m^{\t n} := \big\{V\subseteq j^m: Q_V^{\t{n}}\neq \emptyset
	\big\}.
	\]
	To prove Lemma \ref{lem3}
	we distinguish the following four cases.
	\begin{itemize}
		\item[Case 1.] For every $m\in\omega$ and every $\t{n}\geq\max\limits_{i\leq k-1}
		\{|\sigma_i|\}$,
		$\text{Meetable}_m^{\t n}$ is
		not $2kl_\Psi$-dispersed.
		Moreover, there exists
		$\t{n}\geq\max\limits_{i\leq k-1}
		\{|\sigma_i|\}$ such that
		for every $m\in\omega$,
		$S^u\cap j^m\in\text{Meetable}_m^{\t n}$.

		\item[Case 2.] For every $\t{n}\geq\max\limits_{i\leq k-1}\{|\sigma_i|\}$,
		there exists $ m\in\omega$ such that
		\[
			S^u\cap j^m\notin\text{Meetable}_m^{\t n}.
		\]
		Moreover, for every $m\in\omega$, every $i\in U(c)$, every
		$X=X_0\oplus\cdots\oplus X_{k-1}\in Q$,
		and every $\sigma'\succeq\sigma_i$, with $\sigma'-\sigma_i\subseteq
		X_i\cap A$
		we have that
		\[
			\Psi^{\sigma'}(m)\downarrow\quad\rightarrow\quad
			\Psi^{\sigma'}(m)\cap S^u\cap j^m\ne\emptyset.
		\]

		\item[Case 3.] There exist
		$m\in\omega$, $i\in U(d)$,
		$X=X_0\oplus\cdots\oplus X_{k-1}\in Q$,
		and $\sigma'\succeq \sigma_i$ with $\sigma'-\sigma_i\subseteq X_i\cap A$
		such that $\Psi^{\sigma'}(m)\downarrow$ and $S^u\cap j^m=\emptyset$.

		\item[Case 4.] There exists $m\in\omega$ and $\t{n}\geq\max\limits_{i\leq k-1}
		\{|\sigma_i|\}$
		such that
		$\text{Meetable}_m^{\t n}$ is $2kl_\Psi$-dispersed.
	\end{itemize}

	The four cases cover all the possibilities.
	Indeed, if Case 4 fails, then the first part of Case 1 obtains.
	Then either the second part of Case 1 obtains, or the first part of Case 2 obtains. Then either the second part of Case 2 obtains, or Case 3 obtains.

	We show that
	in Case 1 $S^u$
	admits a bounded enumeration, a
	contradiction;
	in Case 2 the set $A$ would be effectively
	compressible, also a contradiction; in
	Case 3 we construct
	a condition $d'\leq d$ with
	$| U(d')|< |U(d)|$; and
	in Case 4 we construct $d'\leq d$
	such that $U(d')=\emptyset$.
	Therefore the proof is accomplished
	once these are established.
	Now we begin to address each case.

	\begin{lemma}[Case 3 Lemma]\label{lem7}
		If there exist
		$m\in\omega$, $i\in U(d)$,
		$X=X_0\oplus\cdots\oplus X_{k-1}\in Q$,
		and $\sigma'\succeq \sigma_i$ with $\sigma'\subseteq X_i\cap A$
		such that $\Psi^{\sigma'}(m)\downarrow\notin S^u\cap j^m$,
		then there exists a condition $d'$ with identical
		number of parts such that part $i$ of
		$d'$ refines part $i$ of $d$ and
		$d'$ forces $\mathcal{R}_\Psi^u$ on part $i$. Thus,
		$|U(d')| <|U(d)|$.
	\end{lemma}
	\begin{proof}
		Simply extend $\sigma_i$ to $\sigma'$ and keep every
		other
		parts' initial segment.
		That is,
		$d' = (k,\sigma_0,\cdots,\sigma_{i-1},
		\sigma',\sigma_{i+1},\cdots,\sigma_{k-1},
		Q')$ is the desired condition forcing
		$\R_\Psi^u$,
		where $Q' = \big\{X'_0\oplus\cdots\oplus X'_{k-1}\in Q:
		\sigma'-\sigma_i\subseteq X'_i\big\}$
		is clearly nonempty since $X\in Q'$.
	\end{proof}

	\begin{lemma}[Case 1 Lemma]\label{lem6}
		Suppose for every $m\in\omega$, every $\t{n}\geq\max\limits_{i\leq k-1}
		\{|\sigma_i|\}$,
		$\text{Meetable}_m^{\t n}$ is
		not $2kl_\Psi$-dispersed.
		And suppose there exists a
		$\t{n}\geq\max\limits_{i\leq k-1}
		\{|\sigma_i|\}$ such that
		for every $m\in\omega$,
		$S^u\cap j^m\in\text{Meetable}_m^{\t n}$.
		Then $S^u$ admits a bounded enumeration. 
	\end{lemma}
	\begin{proof}
		Fix $\t{n}\geq \max\limits_{i\leq k-1}
		\{|\sigma_i|\}$ such that
		for every $m\in\omega$,
		$S^u\cap j^m\in\text{Meetable}_m^{\t n}$
		and $\text{Meetable}_m^{\t n}$ is not $2kl_\Psi$-dispersed
		(promised by the conditions
		of this lemma).
		The
		set $\text{Meetable}_m^{\t n}$
		is co-c.e uniformly in $m,\t{n}$. Thus, for an arbitrary
		$m\in\omega$, to obtain a $2kl_\Psi$-size
		subset of $j^m$ that has
		nonempty intersection with $S^u\cap j^m$,
		we wait for a time
		$t$ such that $\big\{\ V\subseteq j^m: Q_V^{\t{n}}[t]\neq \emptyset
		\big\}=\{V_0,\cdots,V_{N-1}\}$ is not
		$2kl_\Psi-$dispersed.
		Such a time $t$ must exist since $\text{Meetable}_m^{\t n}$
		is not $2kl_\Psi$-dispersed.
		Let $W_0,\cdots,W_{2kl_\Psi-1}$
		be a partition of $\{0,\cdots,N-1\}$ witnessing
		that $\{V_0,\cdots,V_{N-1}\}$
		is not $2kl_\Psi$-dispersed i.e.,
		\[
			\bigcup\limits_{j\leq 2kl_\Psi-1} W_j =\{0,\cdots,N-1\}
		\] and
		for every $j\leq 2kl_\Psi-1$,
		$W_j\ne\emptyset$ implies $
		\bigcap\limits_{n\in W_j} V_n
		\neq \emptyset$.

		Then, for each $j\leq 2kl_\Psi-1$ with $W_j\ne
		\emptyset$,
		select one element, namely $\rho_j$, from
		$\bigcap\limits_{n\in W_j} V_n$.
		Because $S^u\cap j^m\in \text{Meetable}_m^{\t n}
		\subseteq\big\{\ V\subseteq j^m: Q_V^{\t{n}}[t]\neq \emptyset
		\big\}$, there
		exists some $\t{\jmath}$ such that
		$S^u\cap j^m\in W_{\t{\jmath}}$.
		Therefore $\rho_{\t{\jmath}}\in \bigcap\limits_{n\in W_{\t{\jmath}}} V_n
		\subseteq S^u\cap j^m$.
		Thus,
		$\{\rho_j\}_{j\leq 2kl_\Psi,W_j\ne\emptyset}$ is a
		$2kl_\Psi$-enumeration of $S^u\cap j^m$.
		Finally, the conclusion follows by noticing
		that the procedure is uniform in $m$.
	\end{proof}

	Next, we deal with Case 2.
	\begin{lemma}\label{lem4}
		Suppose that for every $\t{n}\geq\max\limits_{i\leq k-1}\{|\sigma_i|\}$,
		there exists $m\in\omega$ such that
		$S^u\cap j^m\notin\text{Meetable}_m^{\t n}$.
		And suppose that for every $m\in\omega$, every $i\in U(d)$, every
		$X=X_0\oplus\cdots\oplus X_{k-1}\in Q$,
		and every $\sigma'\succeq\sigma_i$ with $\sigma'-\sigma_i\subseteq
		X_i\cap A$
		we have that
		\[
			\Psi^{\sigma'}(m)\downarrow\quad\rightarrow\quad
			\Psi^{\sigma'}(m)\cap S^u\cap j^m\ne\emptyset.
		\]
		Then $A$ is effectively compressible.
	\end{lemma}
	\begin{proof}
		The proof concerns the
		effectiveness of $S^u$.
		Given $N\in\omega$, we compute in the following way
		an $M\in\omega$
		such that $K_U(A\upharpoonright M)\leq M-N+\text{const}$,
		where $\text{const}$ is a constant that does not depend on
		$M,N$.
		To prove this,
		we show that given any $\t{n}$,
		there exists
		$\t{n}'>\t{n}$ computable from $\t{n}$
		and a set $F\subseteq 2^{\t{n}'-\t{n}}$ (computably enumerable uniformly in $\t{n}$),
		such that
		$|F|\leq \frac{1}{2}\cdot 2^{\t{n}'-\t{n}}
		\wedge A\uhr_{\t{n}}^{\t{n}'-1}\in F$.
		Clearly, this is enough for our goal
		since in this way,
		there is a computable
		sequence of integers $\t{n}_0<\t{n}_1<\cdots$
		and a sequence of uniformly c.e. sets $ F_l,l\in\omega$
		such that
		\[
			F_l\subseteq 2^{\t{n}_{l+1}-\t{n}_l},\quad
			|F_l|\leq \frac{1}{2}\cdot 2^{\t{n}_{l+1}-\t{n}_l},\quad\text{and }
			A\uhr_{\t{n}_l}^{\t{n}_{l+1}-1}\in F_l
		\]
		for all $l\in\omega$. Thus
		$K_U(A\upharpoonright \t{n}_N)\leq \t{n}_N-N+\text{const}$.

		Given $\t{n}$, since there exists
		$m$ such that $S^u\cap j^m\notin\text{Meetable}_m^{\t n}$,
		which means $Q_{S^u\cap j^m}^{\t{n}}=\emptyset$,
		then we have that there exists $t^*$ such that
		$Q_{S^u[t^*]\cap j^m}[t^*]=\emptyset$.
		Let $T$ denote
		the pruned co-c.e. tree associated to $Q$.
		\begin{definition}
			A set $\h{A}$ is \emph{diagonal against $S^u$
			at time $t$ on part $i$} if
			there exists $\sigma'\succeq \sigma_i^{\t{n}}$
			with $\sigma'-\sigma_i^{\t{n}}\subseteq \h{A}$
			such that $\Psi^{\sigma'}(m)[t]\downarrow$ and $S^u[t]\cap j^m=\emptyset$.
		\end{definition}
		By a compactness argument,
		$Q_{S^u[t]\cap j^m}[t^*]=\emptyset$
		implies that
		there exists $\t{n}'$ such that
		for every $2$-partition $A_0\oplus A_1$ of
		$\{\t{n},\cdots,\t{n}'-1\}$
		and every $k$-partition $X_0\oplus\cdots\oplus
		X_{k-1}\in T[t]$
		of $\{0,\cdots,\t{n}'-1\}$,
		there exists $c\in \{0,1\}$
		and $i\in U(d)$, such that $X_i\cap A_c$ is diagonal against
		$S^u$ at time $t^*$ on part $i$.
		For any $t'\geq t^*$,
		let $\hat{F}[t']$ be the set of $2$-partitions $A_0\oplus A_1$
		of $\{\t{n},\cdots,\t{n}'-1\}$ such that the following are satisfied.
		\begin{enumerate}
			\item For every $k$-partition
				$X_0\oplus\cdots\oplus X_{k-1}\in T[t']$
				of
				$\{0,\cdots,\t{n}'-1\}$ there exists a
				$i\in U(d) $,
				such that $X_i\cap A_1$ is diagonal against
				$S^u$ at time $t'$ on part $i$;
			\item For each $k$-partition
				$X_0\oplus\cdots\oplus X_{k-1}\in T[t']$
				of
				$\{0,\cdots,\t{n}'-1\}$ and
				each $i\in U(d)$, $X_i\cap A_0$ is not diagonal against $S^u$ at time $t'$ on part $i$.
		\end{enumerate}
		Let $F = \cup_{t'\geq t}\hat{F}[t']$.
		Clearly $\hat{F}[t']$ is computable uniformly in
		$t'\geq t$. Therefore $F$ is computably enumerable.

		\begin{claim}
			For each $t''\geq t'\geq t^*$ and
			each $A_0\oplus A_1\in \hat{F}[t']$, we have that
			$A_1\oplus A_0\notin \hat{F}[t'']$.
		\end{claim}
		\begin{proof}
			Suppose $A_0\oplus A_1\in \hat{F}[t']$.
			By item (1) and since $[T]\ne\emptyset$, there exists
			a $k$-partition
			$X^*_0\oplus\cdots\oplus X^*_{k-1}\in T$
			of
			$\{0,\cdots,\t{n}'-1\}$
			and $i^*\in U(d)$
			such that $X_i^*\cap A_1$ is diagonal against $S^u$
			at time $t'$ on part $i$.
			Therefore, it is impossible that
			for some $t''\geq t'$, for every $k$-partition
			$X_0\oplus\cdots\oplus X_{k-1}\in T[t'']$
			of
			$\{0,\cdots,\t{n}'-1\}$,
			every $i\in U(d)$, $X_i\cap A_1$ is not diagonal against
			$S^u$ at time $t'' $ on part $i$
			with $X^*_0\oplus\cdots\oplus X^*_{k-1},i^*$
			being a witness of this impossibility.
			This impossibility implies $A_1\oplus A_0\notin \hat{F}[t'']$.
		\end{proof}

		Since we have shown that
		for every $t''\geq t'$,
		$\ A_0\oplus A_1\in \hat{F}[t']
		$ implies $ A_1\oplus A_0\notin \hat{F}[t'']$,
		we can conclude
		\[
			A_0\oplus A_1\in F\quad\rightarrow\quad
			A_1\oplus A_0\notin F.
		\]
		So $|F|\leq \frac{1}{2}\cdot 2^{\t{n}'-\t{n}}$.

		Let $A_0^* = A\cap \{\t{n},\cdots,\t{n}'-1\}$, $A_1^* = \{\t{n},\cdots,\t{n}'-1\}-A$.
		It remains to prove the following claim.
		\begin{claim}
			$A^*_0\oplus A^*_1\in F$.
		\end{claim}
		\begin{proof}
			By the definitions of $\t{n}'$ and $t^*$, since $T\subseteq T[t^*]$,
			we have that for every $k$-partition $X_0\oplus\cdots\oplus
			X_{k-1}\in T$
			of $\{0,\cdots,\t{n}'-1\}$,
			there exist $c\in \{0,1\}$ and
			$i\in U(d)$
			such that $X_i\cap A_c^*$ is diagonal against $S^u$
			at time $t^*$ on part $i$.
			Moreover, by the conditions of this lemma,
			we have that
			for each $k$-partition $X_0\oplus\cdots\oplus
			X_{k-1}\in T$
			of $\{0,\cdots,\t{n}'-1\}$,
			each $i\in U(d)$,
			and each $t\in\omega$,
			$X_i\cap A_0^*$ is not diagonal against
			$S^u$ at time $t$ on part $i$.
			These together implies that
			for every $k$-partition $X_0\oplus\cdots\oplus
			X_{k-1}\in T$
			of $\{0,\cdots,\t{n}'-1\}$,
			there exists
			$i\in U(d)$,
			such that $X_i\cap A_1^*$ is diagonal against
			$S^u$ at time $t^*$ on part $i$.
			But
			\begin{eqnarray*}
			\lim\limits_{t'\rightarrow\infty} T[t']&\cap&
			\big\{\text{$k$-partitions of }\{0,\cdots,\t{n}'\}\big\}\\
			= T&\cap& \big\{\text{$k$-partitions of }\{0,\cdots,\t{n}'\}\big\}
			\end{eqnarray*}
			and
			\[
			\lim\limits_{t'\rightarrow\infty}S^u[t']\cap j^m
			=S^u\cap j^m,
			\]
			so there exists a sufficiently large $t''$ such that
			for every $k$-partition $X_0\oplus\cdots\oplus
			X_{k-1}\in T[t'']$
			of $\{0,\cdots,\t{n}'-1\}$,
			there exists
			$i\in U(d)$,
			such that $X_i\cap A_1^*$ is diagonal against $S^u$
			at time $t''$ on part $i$;
			and for every $k$-partition $X_0\oplus\cdots\oplus
			X_{k-1}\in T[t'']$
			of $\{0,\cdots,\t{n}'-1\}$,
			every $i\in U(d)$,
			$X_i\cap A_0^*$ is not diagonal against $S^u$
			at time $t''$ on part $i$.
			Thus $A_0^*\oplus A_1^*\in \hat{F}[t'']\subseteq F$.
		\end{proof}
		This concludes the proof of Lemma \ref{lem4}.
	\end{proof}

	Finally we deal with
	Case 4 which is the key to the proof.

	\begin{lemma}\label{lem5}
		If there exists $m\in\omega$ and $\t{n}\geq\max\limits_{i\leq k-1}
		\{|\sigma_i|\}$
		such that
		$\big\{V\subseteq j^m: Q_V^{\t{n}}\neq \emptyset
		\big\}$ is $2kl_\Psi$-dispersed,
		then there exists $d'\leq d$ such that $d'$
		forces $\mathcal{R}_\Psi^u$.
	\end{lemma}

	\begin{proof}
	We begin by introducing a set operation $\Cross$.

	\begin{definition}[Cross]\label{defcross}
	Given arbitrary $k,N\in\omega$,
	given $N$ many $2k$-partitions
	of $\omega$, namely $\t{X}^n=\t{X}^n_0\oplus\cdots\oplus \t{X}^n_{2k-1}$,
	$n\leq N-1$,
	and a collection $\mathcal{K}$ of nonempty
	subsets of $\{0,\cdots, N-1\}$, we
	define an operation $\Cross$ as follows.
	\[
		\Cross(\t{X}^0,\t{X}^1,\cdots ,\t{X}^{N-1};\mathcal{K}) =
		\bigoplus\limits_{i<2k,K\in\mcal{K}} Y_{i,K}
	,\]
	where $Y_{i,K}
	= \bigcap\limits_{n\in K} \t{X}^n_i $.

	For $N$ many nonempty
	$2k$-partition classes, $Q_0,\cdots,Q_{N-1}$,
	we let
	\begin{eqnarray*}
		\Cross(Q_0,\cdots,Q_{N-1};\mathcal{K})=
		\{\Cross(\t{X}^0,\cdots ,\t{X}^{N-1};
		\mathcal{K})\in 2^\omega:
		\t{X}^n\in Q_n,
		n< N
		\}.
	\end{eqnarray*}
	\end{definition}
	The following claim is easy to verify.
	\begin{claim}\label{fac4}For any
		$N$ many nonempty $\Pi_1^0$
		$2k$-partition classes, $Q_0,\cdots,Q_{N-1}$,
		and any nonempty collection $\mathcal{K}$ of nonempty
		subsets of $\{0,\cdots,N-1\}$,
		\[
			\Cross(Q_0,\cdots, Q_{N-1};
		\mathcal{K})
		\]
		is a nonempty $\Pi_1^0$ class.
	\end{claim}
	Now fix $m$ and $\t{n}$ such that
	$\{V\subseteq j^m:Q^{\t{n}}_V\ne\emptyset\}
	=\{V_0,\cdots,V_{N-1}\}$
	is $2kl_\Psi$-dispersed.
	Let $\mcal{K}$ be the following
	collection of nonempty subsets of $\{0,\cdots,N-1\}$:
	\begin{align}\nonumber
		\mathcal{K} =
		\big\{K\subseteq \{0,\cdots, N-1\}: \{V_n\}_{n\in K}
		\textrm{ is }l_\Psi\text{-dispersed}\big\}.
	\end{align}
	Define a $\Pi_1^0$ class as follows:
	\begin{align}\label{defQ}
		Q' = \Cross(Q^{\t{n}}_{V_0},\cdots,Q^{\t{n}}_{V_{N-1}};\mathcal{K}).
	\end{align}
	The desired condition $d'$ is
	\[
		d'=(2k\cdot|\mathcal{K}|,\mathbf{\sigma}^d,Q'),
	\]
	where $\mathbf{\sigma}^d= \{\sigma^{\t{n}}_{i,K}\}_{i<2k,K\in\mcal{K}}$
	represents the
	corresponding replication of $\sigma^{\t{n}}_0,
	\cdots,\sigma^{\t{n}}_{k-1}$, i.e.,
	$\sigma^{\t{n}}_{i,K} = \sigma^{\t{n}}_{[i/2]}$
	for all $i<2k$, $K\in\mcal{K}$ as in Equation (\ref{pad}).

	\begin{claim}\label{claim0}
		$Q'$ is a partition class of $\omega$.
	\end{claim}
	\begin{proof}
		Fix an arbitrary
		$x\in \omega$ and an arbitrary
		$Y\in Q'$.
		By the definitions of $\Cross$ and $Q'$,
		there exists $\t{X}^n =
		\t{X}^n_0\oplus\cdots\oplus \t{X}^n_{2k-1}
		\in Q^{\t{n}}_{V_n},n\leq N-1$ such that
		(see (\ref{defQ})),
		\[
			Y=\Cross(\t{X}^0,\cdots, \t{X}^{N-1};\mathcal{K}).
		\]
		So $Y = \bigoplus\limits_{i<2k,K\in\mcal{K}}
		Y_{i,K}$ with $Y_{i,K}=
		\bigcap\limits_{n\in K}\t{X}^n_i$.
		Consider the following $2k$-partition of $\{0,\cdots,N-1\}$:
		$W_i = \{n\leq N-1: x\in \t{X}^n_i\}$, $i<2k$.
		If for some $\t{\imath}\leq 2k-1$, $W_{\t{\imath}}\in \mcal{K}$,
		then we are done since this implies
		$x\in Y_{i,W_{\t{\imath}}}$.
		Suppose on the contrary $W_i\notin \mcal{K}$ for all
		$i<2k$.
		By the definition of $\mcal{K}$, we have that for each $i<2k$,
		$\{V_n\}_{n\in W_i}$ is not $l_\Psi$-dispersed.
		By the definition of dispersedness (Definition \ref{disperse}), for each $i<2k$,
		there exists an $l_\Psi$-partition of $W_i$, namely
		$W_{i,0},\cdots,W_{i,l_\Psi-1}$ such that
		$W_{i,l}\ne\emptyset$ implies $\bigcap\limits_{n\in W_{i,l}}V_n \ne\emptyset$.
		But then
		\[
			\{W_{i,l}\}_{i<2k,\, l<l_\Psi}
		\]
		is a
		$2kl_\Psi$-partition of $\{0,\cdots,N-1\}$ such that
		for every $i<2k,l<l_\Psi$,
		$W_{i,l}\ne\emptyset$ implies $\bigcap\limits_{n\in W_{i,l}}V_n \ne\emptyset$,
		a contradiction to the $2kl_\Psi$-dispersedness
		of $\{V_0,\cdots,V_{N-1}\}$.
	\end{proof}

	\begin{claim}
	$d'$ is
	a condition extending $ d$.
	\end{claim}
	\begin{proof}
		By the $2kl_\Psi$-dispersedness of $\{V_0,\cdots,V_{N-1}\}$,
		which implies $l_\Psi$-dispersedness
		of $\{V_0,\cdots,V_{N-1}\}$,
		we have
		$\mathcal{K}\ne\emptyset$ since
		$\{0,\cdots,N-1\}\in\mathcal{K}$.
		By the definition of
		dispersedness, every $K\in\mathcal{K}$ is nonempty.
		So $Q'$ is well-defined.
		Clearly $Q'$ is a $\Pi_1^0$ class
		by Claim \ref{fac4}.
		It is also easy to see
		that $Q'\ne\emptyset$
		by the fact that
		$Q^{\t{n}}_{V_n}\ne\emptyset$ for all $n\leq N-1$
		and Claim \ref{fac4}.
		By Claim \ref{claim0}, $Q'$ is a $2k|\mcal{K}|$-partition
		class. Thus $d'$ is a condition.
		To see that $d'\leq d$,
		note that for every $Y= \bigoplus\limits_{i<2k,K\in\mcal{K}}
		Y_{i,K} \in Q'$,
		every component $Y_{i,K}$ of $Y$, and every $n\in K$, there exists
		$\t{X}^n_0\oplus \cdots \oplus \t{X}^n_{2k-1}\in Q^{\t{n}}_{V_n}$
		such that
		$Y_{i,K}$ is contained in $\t{X}^n_i$.
		But by the definition of $Q^{\t{n}}_{V_n}$, $\t{X}^n_i$,
		for some
		$X_0\oplus\cdots\oplus X_{k-1}\in Q$, $\t{X}^n_i$
		is contained in $X_{[i/2]}\subseteq\omega$.
		Therefore $(\sigma^{\t{n}}_{i,K},Y_{i,K})
		\leq (\sigma_{[i/2]},X_{[i/2]})$.
		Moreover,
		for every $Y=\bigoplus\limits_{i<2k,K\in\mcal{K}} Y_{i,K}\in Q'$,
		$\sigma^{\t{n}}_{i,K}\subseteq Y_{i,K}$.
		This is because for some $X=X_0\oplus\cdots\oplus X_{k-1}\in Q$,
		$Y_{i,K}\subseteq X_{[i/2]}$
		and as set of integers $\sigma^{\t{n}}_{i,K}
		=\sigma_{[i/2]}\subseteq X_{[i/2]}$.
		Thus we have shown that
		part $i,K$ of $d'$
		refine part $[i/2]$ of $d$. Thus $d'\leq d$.
	\end{proof}

	It remains to prove that $d'$
	forces $\mathcal{R}_\Psi^n$.
	It is clear that
	the following claim
	implies that $\Psi^G(m)\uparrow$
	for all $G$ satisfying
	condition $d'$.

	\begin{claim}\label{fac3}
		For any $i<2k,K\in\mcal{K}$,
		any $Y=\bigoplus\limits_{i<2k,K\in\mcal{K}} Y_{i,K}\in Q'$
		any $\sigma'\succeq \sigma^{\t{n}}_{i,K}$ with
		$\sigma'-\sigma^{\t{n}}_{i,K}\subseteq Y_{i,K}$
		we have: $\Psi^{\sigma'}(m)\uparrow$.
	\end{claim}
	\begin{proof}
		By the definitions of $\Cross$ and $Q$,
		there exists $\t{X}^n =
		\t{X}^n_0\oplus\cdots\oplus \t{X}^n_{2k-1}
		\in Q^{\t{n}}_{V_n},n\leq N-1$ such that
		(see (\ref{defQ})),
		$Y_{i,K} = \bigcap\limits_{n\in K}\t{X}^n_i$.
		Suppose that for some $0\leq l'\leq l_\Psi-1$, $\rho_0,\cdots,\rho_{l'}$
		are all the elements in $\Psi^{\sigma'}(m)\downarrow$.
		Consider the following $l_\Psi$ many subsets of
		$K$:
		\[
			W_l=\begin{cases}
			\{n\in K: \rho_l\in V_n\} & \text{if }l\leq l',\\
			\emptyset& \text{if }l'<l<l_\Psi.
			\end{cases}
		\]
		By the definition of $Q_{V_n}^{\t{n}}$,
		$\Psi^{\sigma'}(m)\downarrow\rightarrow \Psi^{\sigma'}(m)\cap V_n\ne\emptyset$
		for all $n\in K$.
		Therefore $W_l,l<l_\Psi$ is an $l_\Psi$-partition
		of $K$.
		Clearly $W_l\ne\emptyset$ implies
		$\bigcap\limits_{n\in W_l} V_n\ne\emptyset$.
		This contradicts the definition of $\mcal{K}$,
		that for every $K\in\mcal{K}$,
		$\{V_n\}_{n\in K}$ is $l_\Psi$-dispersed.
	\end{proof}
		Thus we have finished the proof of Lemma \ref{lem2}.
	\end{proof}

	Lemma \ref{lem4} takes advantage of the constructibility of
	$S^u$ and randomness of $A$,
	while in \cite{liu2015cone} neither is
	needed.
	We cannot replace the noneffective compressibility
	by Schnorr randomness since in Lemma
	\ref{lem4} the set
	$F$ we construct is merely
	uniformly c.e. in $\t{n}$ instead of
	uniformly computable in $\t{n}$.
	We are curious whether Theorem
	\ref{th1} holds for a Schnorr random set $A$.
	\begin{question}
		Is there a Schnorr random set $A$ such
		that every infinite subset of $A$ computes
		some random set, or computes a
		bounded enumeration of some
		tree that does not admit a
		computable bounded enumeration?
	\end{question}
	We guess that the answer is ``no''.

\bibliographystyle{amsplain}
\bibliography{bibliographylogic}

\end{document}